\documentclass[12pt]{amsart}
\usepackage{graphicx}
\usepackage{amsmath,graphics,color}
\usepackage{amsfonts,amssymb}
\usepackage{xypic}
\usepackage{enumerate}
\theoremstyle{plain}

\newtheorem*{theorem*}{Theorem}
\newtheorem*{lemma*} {Lemma}
\newtheorem*{corollary*} {Corollary}
\newtheorem*{proposition*}{Proposition}
\newtheorem*{conjecture*}{Conjecture}
\newtheorem{theorem}{Theorem}[section]
\newtheorem{lemma}[theorem]{Lemma}
\newtheorem*{theorem1*}{Theorem 1}
\newtheorem*{theorem2*}{Theorem 2}
\newtheorem*{theorem3*}{Theorem 3}

\newtheorem{proposition}[theorem]{Proposition}

\theoremstyle{remark}

\newtheorem{example*}{Example}
\newtheorem*{claim}{Claim}

\theoremstyle{definition}

\def\G{\Gamma}

\def\eps{\epsilon}
\def\gl{\mbox{GL}}   \def\Z{\Bbb{Z}} \def\R{\Bbb{R}} \def\C{\Bbb{C}}
\def\N{\Bbb{N}}   \def\ll{\langle} \def\rr{\rangle}
 \def\a{\alpha} \def\g{\gamma}  \def\bp{\begin{pmatrix}}
\def\sm{\setminus} \def\ep{\end{pmatrix}} \def\bn{\begin{enumerate}} 
 \def\rank{\op{rk}}  \def\en{\end{enumerate}}
\def\ba{\begin{array}} \def\ea{\end{array}} 
 \def\S{\Sigma}  \def\a{\alpha}  \def\ti{\tilde}
\def\id{\mbox{id}}   
  
\def\ker{\mbox{Ker}}\def\be{\begin{equation}} \def\ee{\end{equation}} 
   
 \def\hom{\mbox{Hom}}

    \def\rk{\op{rk}}
\def\ct{\C[t^{\pm 1}]} 
\def\op{\operatorname}

\def\GG{\mathcal{G}}

\def\i{\iota}

\def\co{\colon}
\def\e{\epsilon}

\begin{document}
\title{Splittings  of knot groups}
\author{Stefan Friedl}
\address{Mathematisches Institut\\ Universit\"at zu K\"oln\\   Germany}
\email{sfriedl@gmail.com}

\author{Daniel S. Silver}
\address{Department of Mathematics and Statistics\\ University of South Alabama}
\email{silver@southalabama.edu}

\author{Susan G. Williams}
\address{Department of Mathematics and Statistics\\ University of South Alabama}\thanks{The second and third
authors were partially supported by grants \#245671 and \#245615 from the Simons
Foundation.}
\email{swilliam@southalabama.edu}

\date{\today}
\begin{abstract}
Let $K$ be a knot of  genus $g$. If $K$ is fibered,  then it is well known that the knot group $\pi(K)$ splits only over a free group of rank $2g$.
We show that if $K$ is not fibered, then $\pi(K)$ splits over non-free groups of arbitrarily large rank.
Furthermore, if $K$ is not fibered, then $\pi(K)$ splits over every free group of rank at least $2g$. However, $\pi(K)$  cannot split over a  group of rank less than $2g$.
The last statement is proved using the recent results of Agol, Przytycki--Wise and Wise.
\end{abstract}
\maketitle

\section{Introduction}

We start out with a few definitions from group theory.
Let $\pi$ be a group. We say that \emph{$\pi$ splits over the subgroup $B$}
if $\pi$ admits an HNN decomposition with base group $A$ and amalgamating subgroup $B$. More precisely,
$\pi$ splits over the subgroup $B$ if there exists an isomorphism
\[ \pi\xrightarrow{\cong} \ll A,t\,|\, \varphi(b)=tbt^{-1} \mbox{ for all }\, b\in B\rr,\]
where $B \subset A$ are subgroups of $\pi$ and $\varphi\co  B \to A$ is a monomorphism.
In this notation, relations of $A$ are implicit. We will write such a presentation  more compactly as $ \ll A,t\,|\, \varphi(B)=tBt^{-1}\rangle.$

In this paper we are interested in splittings of knot groups.
Given a knot  $K\subset S^3$  we denote the knot group $\pi_1(S^3\sm K)$ by $\pi(K)$.
We denote by $g(K)$ the genus of the knot, the minimal genus of a Seifert surface $\Sigma$ for $K$.
It follows from the Loop Theorem and the Seifert-van Kampen theorem that we can split  the knot group $\pi(K)$
over the free group $\pi_1(\Sigma)$ of rank $2g(K)$. The \emph{rank} $\rk(G)$ of a group $G$ is the minimal size of a set of generators for $G$.

It is well known
that if $K$ is a fibered knot, that is, the knot complement $S^3\sm K$ fibers over $S^1$, then the group $\pi(K)$
splits only over free groups of rank $2g(K)$. (See, for example, Lemma \ref{lem:fibsplit}.)
We show  that this property characterizes fibered knots. In fact, we can say much more.

\begin{theorem}\label{thm:fibsplitintro}
Let $K$ be a non-fibered knot. Then $\pi(K)$ splits over non-free groups of arbitrarily large rank.
\end{theorem}

Neuwirth \cite[Problem~L]{Ne65} asked whether there exists a knot $K$ such that $\pi(K)$  splits over a \emph{free} group
of rank other than $2g(K)$. By the above, such a knot would necessarily have to be non-fibered.
Lyon \cite[Theorem~2]{Ly71}  showed that there does in fact exist a non-fibered genus-one knot $K$ with incompressible Seifert surfaces of arbitrarily large genus. This implies in particular that there exists a knot $K$ for which  $\pi(K)$
splits over free groups of arbitrarily large rank.  We give a strong generalization of this result.

\begin{theorem}\label{freesplitting} 
Let $K$ be a non-fibered knot. Then for any integer $k\geq 2g(K)$ there exists a splitting of $\pi(K)$ over a free group of rank $k$.
\end{theorem}

Note that an incompressible Seifert surface gives rise to a splitting over a free group of \emph{even rank}. The splittings over free groups of \emph{odd rank} in the theorem are therefore not induced by incompressible Seifert surfaces.

Feustel and Gregorac \cite{FG73} showed that if $N$ is an aspherical, orientable $3$-manifold
such that $\pi=\pi_1(N)$ splits over the fundamental group of a \emph{closed} surface $\Sigma\ne S^2$, then
this splitting can be realized topologically by a properly embedded surface.
(More splitting results can be found in  \cite[Proposition~2.3.1]{CS83}.)
The fact that fundamental groups of non-fibered knots can be split over free groups of odd rank shows that the result
of Feustel and Gregorac does not hold for splittings over fundamental groups of surfaces with boundary.

Theorems \ref{thm:fibsplitintro} and \ref{freesplitting} can be viewed as strengthenings of Stallings's fibering criterion.
We refer to Section \ref{section:stallings} for a precise statement.

Our third main theorem shows that Theorem \ref{freesplitting} is optimal.

\begin{theorem}\label{mainthm}\label{mainthm3}
If $K$ is a knot, then $\pi(K)$ does not split over a  group of rank less than $2g(K)$.
\end{theorem}

The case $g(K)=1$ follows from the Kneser Conjecture and work of Waldhausen \cite{Wal68b}, as we show in Section \ref{section:genusone}. However, to the best of our knowledge, the classical methods
 of 3-manifold topology do not suffice to prove  Theorem \ref{mainthm}  in the general case.
We use the recent result  \cite{FV12a} that Wada's invariant  detects the genus of any knot.
This result in turn relies on the seminal work of
 Agol \cite{Ag08,Ag12}, Wise \cite{Wi09,Wi12a,Wi12b}, Przytycki--Wise \cite{PW11,PW12a} and Liu \cite{Liu11}.

Theorem \ref{mainthm} is of interest for several reasons:
\bn
\item It gives a completely group-theoretic chararcterization of the genus of a knot, namely
\[ g(K)=\frac{1}{2}\mbox{min}\{\rk(B)\,|\, \pi(K)\mbox{ splits over the group }B\}.\]
A different group-theoretic characterization was given by Calegari (see the proof of Proposition 4.4 in \cite{Ca09}) in terms of the `stable commutator length' of the longitude. 
\item Theorem \ref{mainthm} fits into  a long sequence of results showing that minimal-genus Seifert surfaces `stay minimal' even if one relaxes some conditions. For example, Gabai \cite{Ga83} showed that the  genus of an \emph{immersed} surface cobounding a longitude of $K$ is at least $g(K)$.
    Furthermore, minimal-genus Seifert surfaces give rise to surfaces of minimal complexity in the 0-framed surgery $N_K$
    (see \cite{Ga87}) and in most $S^1$-bundles over $N_K$ (see \cite{Kr99,FV12b}).

\item Given a closed $3$-manifold $N$ it is obvious that  $\rk(\pi_1(N))$ is a lower bound for the Heegaard genus $g(N)$ of $N$.
In light of Theorem \ref{mainthm} one might hope that this is in an equality; that is, that
$\rk(\pi_1(N))=g(N)$. This is not the case, though, as was shown by various authors (see \cite{BZ84,ScW07} and \cite{Li13}).
\en

The paper is organized as follows.
In Section \ref{section:hnnsplittings} we discuss several basic facts about HNN decompositions of groups.
In Section \ref{section:splitk} we recall that incompressible Seifert surfaces give rise to HNN decompositions of knot groups
and we characterize in Lemma \ref{lem:fibsplit} the splittings of fundamental groups of fibered knots.
In Section \ref{section:52} we consider the genus-one non-fibered knot $K=5_2$. We give explicit examples of splittings of the knot group 
over a non-free group and over the free group $F_3$ of rank 3, and inequivalent splittings of the knot group over $F_2$.

Section \ref{section:splitnonfree}  contains the proof of Theorem  \ref{thm:fibsplitintro}, and in Section
\ref{section:splitfree} we give the proof of Theorem \ref{freesplitting}.  In Section \ref{section:stallings} we show that these two theorems strengthen Stallings's fibering criterion.
In Section \ref{section:genusone} we give a proof of Theorem \ref{mainthm} for genus-one knots.
The proof relies mostly on the Kneser Conjecture and a theorem of Waldhausen. In Section \ref{section:wada} we review the definition of  Wada's invariant of a group. Finally, in Section \ref{section:proof} we prove Theorem \ref{thm:technical}, which combined with the main result of \cite{FV12a} provides a proof of Theorem \ref{mainthm}  for all genera.

We conclude this introduction with two questions. The precise notions are explained in Section \ref{section:hnnsplittings}.

\bn
\item Let $\pi$ be a word hyperbolic group and let $\e\colon \pi\to \Z$ be an epimorphism such that $\ker(\e)$ is not finitely generated.
Does $(\pi,\e)$ admit splittings over (infinitely many) pairwise non-isomorphic groups? (The group $\pi=\pi(K)$ satisfies these conditions if $K$ is a non-fibered knot.)
\item Let $K$ be a non-fibered knot of genus $g$. Does $\pi(K)$ admit (infinitely many) inequivalent splittings over the free group $F_{2g}$ on $2g$ generators?
\en

\subsection*{Conventions and notations.}
All groups are assumed to be finitely presented unless we say specifically otherwise. 
All $3$-manifolds are assumed to be connected, compact and orientable.
Given a submanifold $X$ of a $3$-manifold $N$, we denote by $\nu X\subset N$ an open tubular neighborhood of $X$ in $N$.
Given $k\in \N$ we denote by $F_k$ the free group on $k$ generators.

\subsection*{Acknowledgments.}
The first author wishes to thank the University of Sydney for its hospitality.
We are also very grateful to Eduardo Martinez-Pedroza, Saul Schleimer and Henry Wilton for very helpful conversations.

\section{Hnn-decompositions and splittings of groups}\label{section:hnnsplittings}

\subsection{Splittings of groups}
An \emph{HNN decomposition} of a group $\pi$ is a 4-tuple
$(A, B, t, \varphi)$ consisting of  subgroups $B \le A$ of $\pi$, a \emph{stable letter} $t \in \pi$, and an injective
 homomorphism $\varphi\co  B \to A$, such that the natural inclusion maps induce an isomorphism from $\ll A,t\,|\, \varphi(B)=tBt^{-1}\rr$ to $\pi$.
Alternatively, a HNN-decomposition of $\pi$ can be viewed as an isomorphism
\[ f\co \pi\xrightarrow{\cong} \ll A,t\,|\, \varphi(B)=tBt^{-1}\rr\]
where $\varphi\co B\to A$ is an injective map. We will frequently go back and forth between these two points of view.

We need a few more definitions:
\bn
\item
Given an HNN-decomposition $(A, B, t, \varphi)$
we refer to  the homomorphism $\e\colon \pi\to \Z$ that is given by $\e(t)=1$ and $\e(a)=0$ for $a\in A$ as the \emph{canonical epimorphism}.
\item 
Let $\pi$ be a group and let $\e\in \hom(\pi,\Z)$ be an epimorphism. A \emph{splitting of $(\pi, \e)$ over a subgroup $B$
(with base group $A$)} is an HNN decomposition $(A, B, t, \varphi)$ of $\pi$ such that $\e$ equals the canonical epimorphism. 
With the alternative point of view explained above, a splitting of $(\pi,\e)$ is an isomorphism
\[ f\co \pi\xrightarrow{\cong} \ll A,t\,|\, \varphi(B)=tBt^{-1}\rr\]
such that the following diagram commutes:
\[ \xymatrix{ \pi\ar[dr]_\e \ar[rr]^-{f}&& \ll A,t\,|\, \varphi(B)=tBt^{-1}\rr\ar[dl]^\psi\\
&\Z&}\]
where $\psi$ denotes the canonical epimorphism.
\item Two splittings  $(A, B, t, \varphi)$ and $(A', B', t', \varphi')$ of $(\pi,\e)$  are called \emph{weakly equivalent} if there exists an automorphism $\Phi$ of $\pi$ with  $\Phi(B)=B'$. If $\Phi$ can be chosen to be an inner automorphism of $\pi$, then the two HNN decompositions are said to be \emph{strongly equivalent}.
 \en
We conclude this section with the following well-known lemma of \cite{BS78}. It appears as Theorem B* in \cite{Str84} where an elementary proof can be found. 

\begin{lemma}\label{lem:splitexists}
Let $\pi$ be a finitely presented group and let $\e\in \hom(\pi,\Z)$ be an epimorphism. Then there exists a splitting
\[ f:\pi\xrightarrow{\cong} \ll A,t\,|\, \varphi(B)=tBt^{-1}\rr\]
of $(\pi,\e)$ where $A$ and $B$ are finitely generated.
\end{lemma}

\subsection{Splittings of pairs $(\pi,\e)$ with finitely generated kernel}\label{section:fg}

The following lemma characterizes splittings of pairs $(\pi,\e)$
for which $\ker(\e)$ is finitely generated.

\begin{lemma}\label{lem:splitkerfg}
Let $\pi$ be a finitely presented group, $\e\colon \pi\to \Z$ an epimorphism, and $t$ an element of $\pi$ with $\e(t)=1$.
If $\ker(\e)$ is finitely generated, then  there exists a canonical isomorphism
\[ \pi=\ll B,t\,|\,\varphi(B)=tBt^{-1}\rr\]
where $B:=\ker(\e)$ and  where $\varphi\co B\to B$ is given by conjugation by $t$.
Furthermore, any other  splitting of $(\pi,\e)$
is strongly equivalent to this splitting.
\end{lemma}

\begin{proof}
Let $\pi$ be a finitely presented group and let $\e\colon \pi\to \Z$ be an epimorphism
such that $B=\ker(\e)$ is finitely generated.  We have an exact sequence
\[ 1\to B\to \pi\xrightarrow{\e}\Z \to 0.\]
Let $t\in \pi$ with $\e(t)=1$. The map $n\mapsto t^n$ defines a right-inverse of $\e$, and we see that $B$ is canonically isomorphic to 
 the semi-direct product $\ll t\rr\ltimes B$ where $t^n$ acts on $B$ by conjugation by $t^n$. That is, we have a canonical isomorphism
\[ \pi=\ll B,t\,|\,\varphi(B)=tBt^{-1}\rr.\]

We now suppose that we have another splitting 
 $\pi=\ll C,s\,|\, \psi(D)=sDs^{-1}\rr$
 of $(\pi,\e)$. 
By our hypothesis the group $B=\ker(\e)$ is finitely generated.
On the  other hand, it follows from standard results in the theory of graphs of groups (see \cite{Se80}) that 
\[ \ker(\e)\cong  \cdots C_k*_{D_k} C_{k+1}*_{D_{k+1}} C_{k+2} \cdots,\]
where $C_i=C$ and $D_i=D$ for all $i\in \Z$ and each map $D_i\to C_{i+1}$ is given by $\psi$.

As in \cite{Ne65}, the fact that the infinite free product with amalgamation is finitely generated implies that $C_i=D_i=\psi(D_{i-1})$ for all $i\in \Z$. This, in turn, implies that each $C_i$ and $D_i$ is isomorphic to $D=\ker(\e)$. It is now clear that the identity on $\pi$ already has the desired property relating the two splittings of $(\pi,\e)$. 
\end{proof}

\subsection{Induced splittings of groups}\label{section:anm}

Let
\[ \pi=\ll A,t\,|\, \varphi(B)=tBt^{-1}\rr\]
be an HNN-extension. Given $n\leq m\in \N$ we denote by $A_{[n,m]}$ the result of amalgamating
the groups $t^{i}At^{-i}$, $i=n,\dots,m$ along the subgroups $t^{i}\varphi(B)t^{-i}=t^{i+1}Bt^{-i-1}$, $i=n,\dots,m-1$.
In our notation,
\[ A_{[n,m]}= \langle \ast_{i=m}^n t^i A t^{-i}\mid t^j\varphi(B)t^{-j} = t^{j+1}Bt^{-j-1}\ (j=n, \ldots, m-1) \rr.\]
Given any $k\leq m\leq n\leq l$, we have a canonical map $A_{[m,n]}\to A_{[k,l]}$
which is  a monomorphism (see, for example, \cite{Se80} for details). If $\e\co \pi\to \Z$ is the canonical
 epimorphism, then it is well known that
$\ker(\e)$ is given by the direct limit of the groups $A_{[-m,m]}$, $m\in \N$; that is,
\[ \ker(\e)=\lim_{m\to \infty}A_{[-m,m]}.\]
The following well-known lemma shows that a splitting of a pair $(\pi,\eps)$ gives rise to a sequence of splittings.

\begin{lemma}\label{lem:moresplittings}
Let
\[ \pi=\ll A,t\,|\, \varphi(B)=tBt^{-1}\rr\]
be an HNN-extension.
For any integer $n\ge 0$, let
 \[ \varphi_n: \pi_1(A_{[0,n]})\to A_{[1,n+1]} \]
be the map that is given by conjugation by $t$.
Then the obvious  inclusion maps induce an isomorphism
\[\ll A_{[0,n+1]},t\,|\, \varphi_n(A_{[0,n]})=tA_{[0,n]}t^{-1}\rr\xrightarrow{\i} \pi=\ll A,t\,|\, \varphi(B)=tBt^{-1}\rr.\]
\end{lemma}

\begin{proof}
We write
\[ \G=\ll A_{[0,n+1]},t\,|\, \varphi_n(A_{[0,n]})=tA_{[0,n]}t^{-1}\rr.\]
We denote by $\pi'$ (respectively $\G'$) the kernel of the canonical map from $\pi$ (respectively $\G$) to $\Z$.
It is clear that it suffices to show that the restriction of $\i:\G\to \pi$ to $\pi'\to \G'$ is an isomorphism.

For $i\in \Z$, we  write $A_i:=t^iAt^{-i}$ and $B_i:=\varphi(t^{i+1}Bt^{-i-1})$.
Note that $\G'$ is canonically isomorphic to
\[\cdots  \left(A_0*_{B_0}\cdots*_{B_{n}}A_{n+1}\right)*_{A_1*_{B_1}\cdots*_{B_{n}}A_{n+1}}    \left(A_1*_{B_1}\cdots*_{B_{n+1}}A_{n+2}\right)
*_{A_2*_{B_2}\cdots*_{B_{n+1}}A_{n+2}}
 \cdots ,\]
 and $\pi'$ is canonically isomorphic to
 \[   \cdots *_{B_0}A_{-1}*_{B_{-1}}A_{0}*_{B_0}A_1*_{B_1}*\cdots \]
It is now straightforward to see that $\i$ does indeed restrict to an isomorphism $\G'\to \pi'$.
\end{proof}

Note that the isomorphism in Lemma \ref{lem:moresplittings} is canonical.
Throughout the paper we will therefore make the identification
\[\pi=\ll A_{[0,n+1]},t\,|\, \varphi_n(A_{[0,n]})=tA_{[0,n]}t^{-1}\rr.\]
In the paper we will also write $A=A_{[0,0]}$.

\section{Splittings of knot groups and incompressible surfaces}\label{section:splitk}
Now let  $K\subset S^3$ be a knot, that is, an oriented embedded simple closed curve in $S^3$.
We write $X(K):=S^3\sm \nu K$ and
\[ \pi(K):=\pi_1(X(K))=\pi_1(S^3\sm \nu K).\]
The orientation of $K$ gives rise to a canonical epimorphism $\e_K\co \pi(K)\to \Z$  sending the oriented meridian to 1.

Let  $\Sigma$ be a Seifert surface of genus $g$ for $K$; that is, a connected, orientable, properly embedded surface $\Sigma$ of genus $g$
in $X(K)$ such that $\partial \Sigma$ is an oriented longitude for $K$.
Note that $\Sigma$ is dual to the canonical epimorphism $\e$.

Suppose that $\Sigma$ is incompressible.
(Recall that a surface $\Sigma$ in a $3$-manifold $N$ is called \emph{incompressible} if the inclusion-induced map $\pi_1(\Sigma)\to \pi_1(N)$ is injective.)
We  pick a tubular neighborhood $\S\times [-1,1]$. The manifold
$X(K)\sm  \S\times (-1,1))$ is the result of \emph{cutting
along $\S$}. The Seifert--van Kampen theorem gives us
a  splitting
\[ \pi_1(X(K))= \ll \pi_1(X(K)\sm  \S\times (-1,1)),t\,\,|\,\, \varphi(\pi_1(\S\times -1)=t\pi_1(\S\times 1)t^{-1}\rr\]
of $(\pi(K),\e_K)$, where $\varphi$ is induced by the canonical homeomorphism $\S\times -1\to \S\times 1$.
We thus see that  $\pi(K)$ splits over the free group $\pi_1(\Sigma)$ of rank $2g$.

Given a knot $K\subset S^3$, we denote by $g=g(K)$ the minimal genus of a Seifert surface. It follows from the Loop Theorem
(see, for example, \cite[Chapter~4]{He76})  that
 a Seifert surface of minimal genus is incompressible. Hence $\pi(K)$ splits over a free group of rank $2g(K)$.
 \medskip
 
If two incompressible Seifert surfaces of a knot $K$ are isotopic, then it is clear that the corresponding splittings of $\pi(K)$ are strongly equivalent.
There are many examples of knots that admit non-isotopic minimal genus Seifert surfaces; see e.g.  \cite{Ly74b,Ei77a,Ei77b,Al12,HJS13}.
We expect that these surfaces give rise to splittings that are not strongly equivalent.
\medskip
 
On the other hand, if a knot is fibered, then it admits a unique minimal genus Seifert surface up to isotopy (see e.g. \cite[Lemma~5.1]{EL83}). It is therefore perhaps not entirely surprising that $\pi(K)$ admits a unique splitting up to strong equivalence.
More precisely, we have the following well-known
 lemma, which is originally due to Neuwirth \cite{Ne65}.
 
\begin{lemma} \label{lem:fibsplit}
Let $K$ be a fibered knot of genus $g$ with fiber $\Sigma$. Then any splitting of $\pi(K)$ is strongly equivalent to 
\[ \ll \pi_1(X(K)\sm  \S\times (-1,1)),t\,|\, \varphi(\pi_1(\S\times -1)=t\pi_1(\S\times 1) t^{-1}\rr.\]
In particular $\pi(K)$ only splits over the  free group of rank $2g$.
\end{lemma}
 
\begin{proof}
If $\S$ is a fiber surface for $X(K)$, then the infinite cyclic cover of $X(K)$ is diffeomorphic to $\S\times \R$. Put differently, $\ker(\e_K)\cong \pi_1(\S)$
which implies in particular that $\ker(\e_K)$ is finitely generated.
The lemma is now a straightforward consequence of Lemma \ref{lem:splitkerfg}.
\end{proof}
 
\section{Splitting of the knot group for $K=5_2$}\label{section:52}
In this section we give several explicit splittings of the knot group $\pi(K)$ where   $K=5_2$, the first non-fibered knot in the Alexander-Briggs table. 
We construct:
\bn
\item three splittings of $\pi(5_2)$ over the free group $F_2$, no two being weakly equivalent;
\item a splitting of $\pi(5_2)$ over the free group $F_3$ on three generators;
\item a splitting of $\pi(5_2)$ over a non-free group.
\en
Note that neither the second nor the third  splitting is induced by an incompressible surface. We will also see that at least two of the three splittings over $F_2$ are not  induced by an incompressible surface.

Since $K$ is a knot of genus one, a minimal-genus Seifert surface gives rise to a splitting of $\pi(K)$ over a free group of rank 2. In the following we will consider an explicit splitting that comes from a Wirtinger presentation of the knot group:
\[ \pi(K) =\ll  a, b,t\,|\, tat^{-1}=b,\, tb^{-1}ab^{-1}t^{-1}=(b^{-1}a)^2\rr.\]
Here the knot group has an HNN decomposition $(A, B, t, \varphi)$,
where $A$ is the free group on $a, b$ while
$B$ is the subgroup freely generated by $a$ and $b^{-1}ab^{-1}$. The isomorphism $\varphi$ sends $a \mapsto b$ and $b^{-1}ab^{-1}\mapsto (b^{-1}a)^2$.
For the remainder of this section we identify $\pi(K)$ with $\ll A,t\,|\, \varphi(B)=tBt^{-1}\rr$.

\begin{proposition}\label{prop:52splittings}
Consider the splittings:
\[ \ba{rcl} \pi(K)& =& \ll A,t\,|\, \varphi(B)=t Bt^{-1}\rr, \\
 \pi(K)&=&\ll A_{[0,1]},t\,|\, \varphi_1(A)=tAt^{-1}\rr,\\
 \pi(K)&=&\ll A_{[0,2]},t\,|\, \varphi_2(A_{[0,1]})=tA_{[0,1]}t^{-1}\rr\ea \]
where the latter two splittings are provided by Lemma \ref{lem:moresplittings}.
Then the following hold.
\begin{enumerate}[(i)]
\item Each is a splitting over a free group of rank two.
\item No two of the splittings of $(\pi(K),\e_K)$ are weakly equivalent.
\item At least two of the splittings are not induced by an incompressible Seifert surface.
\end{enumerate}
\end{proposition}

In the proof of Proposition \ref{prop:52splittings} we will make use of the following lemma which is perhaps also of independent interest.

\begin{lemma}\label{lem:propersubgroup}
Let $M$ be a hyperbolic $3$-manifold with empty or toroidal boundary, and let $G$ be a  subgroup of $\pi:=\pi_1(M)$.
If $f\co \pi\to \pi$ is an automorphism with $f(G)\subset G$, then $f(G)=G$.
\end{lemma}

We do not know whether the conclusion of the lemma holds for any $3$-manifold.

\begin{proof}
Let $f\co \pi\to \pi$ be an automorphism with $f(G)\subset G$.
Since $M$  is hyperbolic, it is a consequence of the Mostow Rigidity Theorem that the group of outer automorphisms of $\pi$ is finite.
(See, for example, \ \cite[Theorem~C.5.6]{BP92} and \cite[p.~213]{Jo79} for details.)
Consequently, there exists a positive integer $n$ and an element $x\in \pi$ such that $f^n(G)=xGx^{-1}$.
It follows from   \cite[Theorem~4.1]{Bu07} that $f^n(G)=G$. The assumption that $f(G)\subset G$ implies inductively that $f^n(G)\subset f(G)$. Hence $f(G)=G$.
\end{proof}

We can now turn to the proof of Proposition \ref{prop:52splittings}.

\begin{proof}[Proof of Proposition \ref{prop:52splittings}]
It is clear that the first and the second splitting are over a free group of rank two.
It remains to show that $A_{[0,1]}$ is a free group of rank two.
First note that
\[A_{[0,1]}\cong \ll a_0, b_0, a_1, b_1\,|\, a_1=b_0,\, b_1^{-1} a_1b_1^{-1}=(b_0^{-1}a_0)^2\rr,\]
where $a_i$ and $b_i$ denote $t^iat^{-i}$ and $t^ibt^{-i}$, respectively. Using the first relation to eliminate the generator $b_0$, we obtain
$A_{[0,1]}\cong  \ll a_0, a_1, b_1\,|\, r \rr,$ where $r= (a_1^{-1}a_0)^2 b_1a_1^{-1} b_1.$
We let $c = a_1^{-1}a_0$ and $d = b_1a_1^{-1}$. Clearly $\{c, d, r\}$ is a basis for the free group on $a_0, a_1, b_1$. Hence $A_{[0,1]}\cong \ll c, d, r\, |\, r \rr \cong \ll c, d\, |\, \rr$ is indeed a free group of rank 2.
This concludes the proof of (i).

We turn to the proof of (ii).
Since $K$ is not fibered it follows 
from Stallings's theorem (see Theorem \ref{thm:st62})
that $\ker(\e_K)= \lim_{k\to \infty} A_{[-k,k]}$ is not finitely generated.
It follows that easily that for any $l\geq k$ the map $A_{[0,k]}\to A_{[0,l]}$ is a proper inclusion.
In particular, we have proper inclusions
$A \varsubsetneq A_{[0,1]} \varsubsetneq A_{[0,2]}$. Since $S^3\sm \nu K$ is hyperbolic, the desired statement now follows from
Lemma \ref{lem:propersubgroup}.

We prove (iii). It is well known (see, for example,\ \cite{Ka05}) that any two minimal-genus Seifert surfaces of $5_2$ are isotopic.
This implies, in particular, that any two splittings of $\pi(K)$ induced by  minimal-genus Seifert surfaces are strongly equivalent.
It follows from (ii) that at least two of the three splittings are not induced by a minimal genus Seifert surface.
\end{proof}

\begin{figure}[h]
\begin{center}
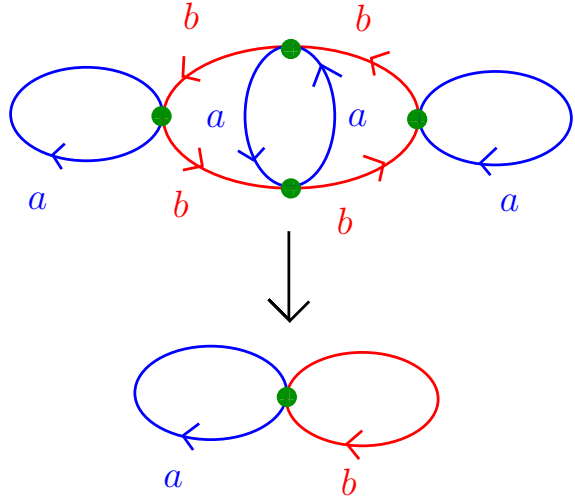
\caption{Covering graph.}\label{Graph}
\end{center}
\end{figure}

We show that $\pi(K)$ admits a splitting over a free group of rank 3.
In order to do so we note that there exists a canonical isomorphism
\be \label{equ1} \ba{rcl} &&\ll  a, b,t\,|\, tat^{-1}=b,\, tb^{-1}ab^{-1}t^{-1}=(b^{-1}a)^2\rr\\
&\cong&\ll a, b,c,t\,|\, tat^{-1}=b,\, tb^{-1}ab^{-1}t^{-1}=(b^{-1}a)^2, tb^{-2}ab^{-2}t^{-1}=c\rr.\ea \ee
Let  $A'$ be the free group generated by $a,b,c$. Let $B'$ be the subgroup of $A'$ generated by $a,b^{-1}ab^{-1},b^{-2}ab^{-2}$.
The fundamental group of the covering graph in Figure \ref{Graph} is free on $a, b^{-1}ab^{-1}, b^{-1}a^2b, b^{-2}ab^{-2}$, and $b^4$, and so $B'$ is a free rank-3 subgroup of $A'$. The elements $b, b^{-1}ab^{-1}a, c$ of $A'$ also generate a free group of rank 3, since they are free in the abelianization of $A'$.
There exists therefore a unique homomorphism  $\varphi':B'\to A'$ such that
$\varphi'(a)=b$, $\varphi'(b^{-1}ab^{-1})=b^{-1}ab^{-1}a$ and $\varphi'(b^{-2}ab^{-2})=c$.
It follows that $\varphi'$ is in fact a monomorphism.
Hence from (\ref{equ1}),
\[ \ll A',t\,|\, tB't^{-1}= \varphi(B')\rr\]
defines a splitting of $\pi(K)$ over the free group $B'$ of rank three.

Finally we give an explicit splitting of $\pi(K)$  over a subgroup that is not free.
Recall that by Lemma \ref{lem:moresplittings} the group $\pi(K)$ admits an HNN decomposition
with the HNN base $A_{[0,2]}$ defined as the amalgamated product of $A, tAt^{-1}$ and $t^2At^{-2}$.
It suffices to prove the following claim.

\begin{claim}
The group $A_{[0,2]}$ is not free.
\end{claim}

Note that $A_{[0,2]}$ has the presentation
\[\ll a_0, b_0,  a_1, b_1, a_2, b_2\,|\, a_1=b_0,\, b_1^{-1}a_1b_1^{-1}=(b_0^{-1}a_0)^2,\, a_2=b_1, \, b_2^{-1}a_2b_2^{-1}=(b_1^{-1}a_1)^2\rr.\]
Using the first and third relations, we eliminate the generators $b_0$ and $b_1$. Thus
\[A_{[0,2]} \cong \ll a_0, a_1, a_2, b_2\,|\, r_1,\, r_2 \rr,\]
where $r_1=(a_1^{-1}a_0)^2a_2a_1^{-1}a_2$ and
$r_2 = (a_2^{-1}a_1)^2 b_2 a_2^{-1}b_2$.

Let $\e=a_1^{-1}a_0$ and $f=a_2a_1^{-1}$.
 One checks that
$\{\e, f, r_1, b_2\}$  is a basis for the free group $\ll a_0, a_1, a_2, b_2\,|\, \rr.$ Using the substitutions
\[ a_0=f^{-2}\e^{-2}r_1\e,\, a_1=f^{-2}\e^{-2}r_1 \mbox{ and } a_2=f^{-1}\e^{-2}r_1,\]
 we see
\[A_{[0,2]} \cong \ll \e, f, b_2 \,|\, r_2 \rr \cong \ll \e, f, b_2 \,|\, f^{-2}\e^{-2}(b_2 \e^2) f(b_2\e^2)\rr.\]
We perform two more changes of variables. First we  let $g = b_2\e^2$ and eliminate $b_2$ to obtain
\[A_{[0,2]} \cong  \ll \e, f, g \,|\, \e^{-2}(gf)^2f^{-3} \rr, .\]
Second, we let $h=gf$ and we eliminate $g$:
\[A_{[0,2]} \cong  \ll \e, f, h \,|\, \e^{-2}h^2 = f^3 \rr.\]
We thus see that $A_{[0,2]}$ is a free product of two free groups amalgamated over an infinite cyclic group.
By Lemma 4.1 of \cite{BF94} (see Example 4.2), if the group $A_{[0,2]}$ is free, then
either $\e^{-2}h^2$ or $f^3$ is a basis element in its respective factor.
Since neither element is a basis element (seen for example by abelianizing),
the group $A_{[0,2]}$ is not free. This concludes the proof of the claim.

\section{Splittings of  fundamental groups of non-fibered knots over non-free groups}\label{section:splitnonfree}

In Section \ref{section:52} we  saw that we can split the knot group $\pi(5_2)$ over a group that is not free.
We will now see that this example can be greatly generalized.
We recall the statement of our first main theorem. 

\begin{theorem}\label{thm:fibsplit}
If $K$ is a non-fibered knot, then $\pi(K)$ admits splittings over \emph{non-free} subgroups of arbitrarily large rank.
\end{theorem}

\begin{proof}
Let $\Sigma\subset X(K)$ be a Seifert surface of minimal genus.
We write $A=\pi_1(X(K)\sm \Sigma\times (-1,1)$ and $B=\pi_1(\S\times -1)$, and we consider the corresponding splitting
	\[ \pi(K)=\ll A,t\,|\, \varphi(B)=tBt^{-1}\rr \]
of $(\pi(K),\e_K)$ over $\pi_1(\S)$.
Given $n\leq m $ we consider, as in Section \ref{section:anm}, the group
\[ A_{[n,m]}= \langle \ast_{i=m}^n t^i A t^{-i}\mid t^j\varphi(B)t^{-1} = t^{j+1}Bt^{-j-1}\ (j=n, \ldots, m-1) \rr.\]
By  Lemma \ref{lem:moresplittings}  the group $\pi(K)$  splits over the  group $A_{[0,n]}$ for any non-negative integer $n$.

\begin{claim}
There exists an integer $m$ such that $A_{[0,n]}$ is not a free group for any $n\geq m$.
\end{claim}

As we pointed out in Section \ref{section:anm}, we have an isomorphism
\[ \ker(\e_K\co \pi(K)\to \Z)\cong \lim_{k\to \infty}A_{[-k,k]}\]
where the maps $A_{[-l,l]}\to A_{[-k,k]}$ for $l\leq k$ are monomorphisms. 
It follows from \cite[Theorem~3]{FF98} that $\ker(\e_K)$ is not  locally free; that is,  there exists a finitely generated subgroup of $\ker(\e_K)$  which is not a free group.
But this implies that there exists $k\in \N$ such that $A_{[-k,k]}$ is not a free group.
We have a canonical isomorphism $A_{[-k,k]}\cong A_{[0,2k]}$, and for any $n\geq 2k$ we have a canonical monomorphism
$A_{[0,2k]}\to A_{[0,n]}$. It now follows that  $A_{[0,n]}$ is not a free group for any $n\geq 2k$.
This concludes the proof of the claim.

To complete the proof of Theorem \ref{thm:fibsplit} it remains to prove the following claim:

\begin{claim}
Writing $H_n:=A_{[0,n]}$  we have
\[ \lim_{n\to \infty} \rk(H_n)=\infty.\]
\end{claim}

Since $\S\subset X(K)$ is not a fiber it follows from \cite[Theorem~10.5]{He76} that
there exists an element $g\in A\smallsetminus B$.
By work of Przytycki--Wise (see \cite[Theorem~1.1]{PW12b}) the subgroup
 $B=\pi_1(\S\times -1)\subset \pi(K)$ is separable. This implies, in particular, that there exists an epimorphism
 $\a\co \pi(K)\to G$ onto a finite group $G$ such that
$\a(g)\not\in \a(B)$.  Then
\[ D:=\a(B)\varsubsetneq C:=\a(A).\]
Given $n\in \N$ we denote by $\a_n$ the restriction of $\a$ to $H_n\subset \pi(K)$ and we write $G_n:=\a(H_n)$.

Note that in
\[ H_n=A_0*_{B_0}\dots*_{B_{n-1}}A_{n}\]
the groups $A_i$, viewed as subgroups of $\pi(K)$, are conjugate.
It follows
that the groups $\a_n(A_i)$ are conjugate in $G$. In particular, each of the groups $\a_n(A_i)$ has order $|C|$.
The same argument shows that each of the groups $\a_n(B_i)$ has order $|D|$.
Standard arguments about fundamental groups of graphs of groups (see, for example, \cite{Se80}) imply that
$\ker(\a_n:H_n\to G_n)$ is the fundamental group of a graph of groups, where the underlying graph $\ti{\GG}$
is a connected graph with  $(n+1)\cdot |G_n|/|C|$ vertices and $n\cdot |G_n|/|D|$ edges.
From the Reidemeister-Schreier theorem (see, for example, \cite[Theorem~2.8]{MKS76} and from the fact that
$\ker(\a_n:H_n\to G_n)$ surjects onto $\pi_1(\ti{\GG})$ it then follows that
\[ \ba{rcl}\rk(H_n)&\geq &\frac{1}{|G_n|}\rk(\ker(\a_n:H_n\to G_n))\\
&\geq&\frac{1}{|G_n|}\rk(\pi_1(\ti{\GG}))\\
&=& \frac{1}{|G_n|}\big(n\cdot |G_n|/|D|-(n+1)\cdot |G_n|/|C|+1\big)\\
&\geq &(n+1)\left(\frac{1}{|D|}-\frac{1}{|C|}\right).\ea \]
But this sequence diverges to $\infty$ since $|D|<|C|$.
\end{proof}

\section{Splittings of  fundamental groups of non-fibered knots over free groups}\label{section:splitfree}

\subsection{Statement of the theorem}

Lyon \cite[Theorem~2]{Ly71}
 showed that there exists a non-fibered knot $K$ of genus one that admits incompressible
Seifert surfaces of arbitrarily large genus  (see also \cite{Sce67,Gu81,Ts04} for related examples).
By the discussion in Section \ref{section:splitk}, this implies that $\pi(K)$ splits over free groups of arbitrarily large rank.

Splitting along incompressible Seifert surfaces is a convenient way to produce knot group splittings.
Yet there are many non-fibered knots that have unique incompressible Seifert surfaces (see, for example, \cite{Wh73,Ly74a,Ka05}).
For such a knot, Seifert surfaces gives rise to only one type of knot group splitting.

In Section \ref{section:52} we saw  an example of a splitting of a knot group over a free group that is not induced by an embedded surface.
We generalize the example in our second main theorem.  We recall the statement.

\begin{theorem}\label{thm:splitfreelarge}
Let $K$ be a non-fibered knot. Then for any integer $k\geq 2g(K)$ there exists a splitting of $\pi(K)$ over a free group of rank $k$.
\end{theorem}

The key to extending the result in Section \ref{section:52} is the following theorem, which we will prove in the next subsection.

\begin{theorem}\label{thm:extendfree}
Let $K$ be a non-fibered knot. Then there exists a Seifert surface $\S$ of minimal genus such that for a given base point $p\in \S=\S\times 0$
there exists a nontrivial element $g \in \pi_1(S^3 \sm \S \times (0,1),p)$ such that the subgroup of $\pi(K)$ generated by $\pi_1(\S\times 0,p)$ and $g$ is the free product of $\pi_1(\S\times 0,p)$ and the infinite cyclic group $\ll g \rr$.
\end{theorem}

Theorem \ref{thm:splitfreelarge} is now a  consequence of Theorem  \ref{thm:extendfree} and the following proposition about HNN decompositions.

\begin{proposition} Assume that $(\pi, \e)$  splits over a free group $F$ of rank $n$
with base group $A$. If there exists an  element $g \in A$ such that the subgroup of $\pi$ generated by $F$ and $g$ is the free product $F * \ll g \rr$, then $(\pi, \e)$ splits over free groups of every rank greater than $n$. \end{proposition}

\begin{proof}
By hypothesis we can identify 
$\pi$ with
\[\ll A, t \mid \varphi(x_i) = tx_it^{-1}\ (1 \le i \le n) \rr,\]
where $x_1,\ldots, x_n$ generate the group $F$ and where $\e$ is given by  $\e(t)=1$ and $\e(A)=0$.

 The kernel of the second-factor projection
$F*\ll g \rr \to \ll g\rr=\Z$ is an infinite free product $*\{g^iFg^{-i} \mid i \in \Z\}$. Let $l$ be any positive integer. Choose a nontrivial element $z \in F$ and define $z_i = g^izg^{-i}$, for  $1 \le i \le l$. Then
$F'= \ll F, z_1, \ldots, z_l \rr$ is a free subgroup of $F*\ll g\rr$  with rank $n + l$.
By hypothesis $F'$ is then also a free subgroup of $A$ of rank $n+l$.

Note that $\pi$ is canonically isomorphic to
\[\ll A,  c_1, \dots,c_l,  t \mid  \varphi(x_i) = tx_it^{-1}, c_j = tz_jt^{-1} (1\le i \le n, 1\le j \le l) \rr.\]
We denote by $A'$ the free product
 of  $A$ and $\ll c_1,\dots,c_l\rr$,
and we denote by $\varphi'$
the  unique homomorphism
\[ \varphi'\co  F'=F*\ll z_1, \ldots, z_l \rr \to A'=A*\ll c_1,\dots,c_l\rr \]
that extends $\varphi$ and that maps each $z_j$ to $c_j$. Since $\varphi'$ is the free product of two isomorphisms, it is also an isomorphism.
We then have a canonical isomorphism
\[ \pi \cong \ll A',t\mid \varphi'(F')=tF't^{-1}\rr.\]
We have thus shown that $(\pi, \e)$ splits over the free group $F'$ of rank $n+l$.
\end{proof}

\subsection{Proof of Theorem \ref{thm:extendfree}}

To prove Theorem \ref{thm:extendfree} we will need to discuss the JSJ pieces of knot complements. (See \cite{AFW12} for exposition about JSJ decompositions.) 
It is therefore convenient to generalize a few notions for knots to more general 3-manifolds.

Given a $3$-manifold $N$, we can associate to each class $\e\in H^1(N;\Z)$ its Thurston norm $x_N(\e)$, which is defined as the minimal `complexity'
of a surface dual to $\e$.  We say that a class $\e\in H^1(N;\Z)$ is \emph{fibered} if there exists a fibration $p\co N\to S^1$ such that the induced
map $p_*\co \pi_1(N)\to \pi_1(S^1)=\Z$ agrees with $\e\in H^1(N;\Z)=\hom(\pi_1(N),\Z)$.
It is well known that given a non-zero $d\in \Z$, the  class $\e$ is fibered if and only if $d\e$ is fibered.
Note that given a non-trivial knot $K\subset S^3$ we have $x_{X(K)}(\e_K)=2g(K)-1$, and $\e_K$ is fibered if and only if $K$ is fibered.
We refer to \cite{Th86} for background and more information.

We will need the following theorem, which in particular implies Theorem \ref{thm:extendfree} in the case that $S^3\sm \nu K$ is hyperbolic.

\begin{theorem}\label{thm:extendfreehyp}
Let $N$ be a hyperbolic $3$-manifold  and let $\S$ be a properly embedded,
connected Thurston norm-minimizing surface that is not a fiber surface.
We write  $M = N \sm \S\times (0,1)$ and we pick a base point $p$ on $\S\times 0=\S$. Then there exists a nontrivial element $g\in \pi_1(M,p)$ such that the subgroup of $\pi_1(M,p)$ generated by $\pi_1(\S,p)$ and $g$ is the free product of $\pi_1(\S,p)$ and $\langle g \rangle$.
\end{theorem}

\begin{proof}
Let $N$ be a hyperbolic $3$-manifold. We denote by $T_1,\dots,T_k$  the boundary components of $N$. Let $\S$ be a properly embedded,
connected Thurston norm-minimizing surface that is not a fiber surface.
We write  $M = N \sm \S\times (0,1)$ and we pick a base point $p$ on $\S\times 0=\S$.
We now take all fundamental groups with respect to this base point.
It follows again from the Loop Theorem and the fact that $\S$ is Thurston norm-minimizing that  the inclusion-induced map $\G:=\pi_1(\S)\to \pi_1(M)$ is a monomorphism.
We will henceforth view $\G=\pi_1(\S)$ as a subgroup of $\pi_1(M)$.

We first suppose that $\S$ hits all boundary components of $N$.
Since $\S$ is not a fiber surface, it  follows
from  the Tameness Theorem of  Agol \cite{Ag04} and Calegari--Gabai \cite{CG06}
 that $\pi_1(M)$ is word-hyperbolic and that $\G=\pi_1(\S)$ is a quasi-convex subgroup of $\pi_1(M)$.
(We refer to  \cite[Sections~14~and~16]{Wi12a} for more details.)
It then follows from work of Gromov \cite[5.3.C]{Gr87} (see also \cite[Theorem~1]{Ar01})
that there exists an element $g\in \pi_1(M)$ such that the subgroup of $\pi_1(M)$ generated by $\G$ and $g$ is in fact the free product
of $\G$ and $\langle g \rangle$.

We now suppose that there exists a boundary component $T_i$ that is not hit by $\S$.
We  pick a path in $M$ connecting $T_i$ to the chosen base point and we henceforth view $\pi_1(T_i)$ as a subgroup of $\pi_1(M)$.
Note that $\pi_1(N)$ is hyperbolic relative to the subgroups $\pi_1(T_1),\dots,\pi_1(T_k)$.
Since $\S$ is not a fiber surface, it follows from the Tameness Theorem  and from work of Hruska \cite[Corollary~1.3]{Hr10} that $\G$ is a relatively quasi-convex subgroup of $\pi_1(N)$. Since $\G$ is a non-abelian surface group we can find an element $g\in \G$
such that $\ll g\rr\cap \pi_1(T_i)$ is trivial.
We see again from the Tameness Theorem  that $\ll g\rr$ is a relatively quasi-convex subgroup of $\pi_1(N)$.

Summarizing, we have shown that $\pi_1(\S)$ and $\ll g\rr$ are two relatively quasi-convex subgroups of $\pi_1(N)$ which have trivial intersection
with the parabolic subgroup $\pi_1(T_i)$.
It now follows from Martinez-Pedroza \cite[Theorem~1.2]{MP09} that  there exists a $h\in \pi_1(T_i)$ such that the subgroup of $\pi_1(N)$ generated by
$\G$ and $hgh^{-1}$ is the free product of $\G$ and $\ll hgh^{-1}\rr$. The proposition now follows from the observation that
according to our choices, both $\G$ and $\ll hgh^{-1}\rr$ lie in $\pi_1(M)$.
\end{proof}

We can now prove Theorem \ref{thm:extendfree}. 
For the reader's convenience we recall the statement.

\begin{theorem}\label{thm:infinitecyclic}
Let $K$ be a non-fibered knot. Then there exists a Seifert surface $\S$ of minimal genus
and a nontrivial element $g \in \pi_1(S^3 \sm \S \times (0,1))$ such that the subgroup generated by $\pi_1(\S\times  0)$ and $g$ is the free product of $\pi_1(\S\times 0)$ and the infinite cyclic group $\ll g \rr$.
\end{theorem}

\begin{proof}
Let $K$ be a non-fibered knot. We write $X=S^3\sm \nu K$. We denote by $X_v,v\in V$, the JSJ components,
and we denote by $T_\e, \e\in E$, the JSJ tori of $X$.
We let $\e \in H^1(X;\Z)=\hom(H_1(X);\Z),\Z) \cong \Z$ be the generator that corresponds to the canonical homomorphism $\e_K\colon H_1(X;\Z)\to \Z$.
For each $v\in V$, we denote by $\e_v\in H^1(X_v;\Z)$ the restriction of $\e$ to $X_v$. 

The pair $(V,E)$ has a natural graph structure, since each JSJ torus cobounds two JSJ components. Since $X$ is a knot complement,
this graph is a based tree, where the base is  the vertex $b\in V$ for which $X_b$ contains the boundary torus.
We now denote by $T_b$ the boundary torus of $X$, and for each $v\ne b$ we denote by $T_v$ the unique JSJ torus which is a boundary component of $X_v$ and which separates $X_v$ from $X_b$.

\begin{claim}
There exists an element  $w\in V$ such that $X_w$ is hyperbolic and such that $\e_w\in H^1(X_w;\Z)$ is not a fibered class.
\end{claim}

We say that a vertex $v\in V$ is non-fibered if $\e_v\in H^1(X_v;\Z)$ is not a fibered class. Since $\e=\e_K$ is by assumption not fibered,  it follows from \cite[Theorem~4.2]{EN85} that
some vertex  is not fibered.
Let $w\in V$ be a non-fibered vertex of minimal distance to $b$.

Note that if $v\in V$ is fibered and if $\e_v$ is non-trivial, then the restriction of $\e_v$ to any boundary torus is also non-trivial.
Since $\e_b$ is non-trivial and since $w\in V$ is a non-fibered vertex of minimal distance to $b$, we conclude that the restriction of $\e_w$
to $T_w$ is non-trivial.

It follows from the Geometrization Theorem and from  \cite[Lemma~VI.3.4]{JS79} that $X_w$ is one of the following:
\bn
\item the exterior of a torus knot;
\item a `composing space', that is, a product $S^1\times W_n$, where $W_n$ is the result of removing $n$ open disjoint disks from $D^2$;
\item a `cable space', that is,  a manifold obtained from a solid torus $S^1\times D^2$  by removing an open regular neighborhood in
$S^1 \times \op{Int}(D^2$) of a simple
closed curve $c$ that lies in a torus $S^1\times s $, where $s\subset \op{Int}(D^2)$ is a simple closed curve
and $c$ is  non-contractible  in $S^1\times D^2$;
\item a hyperbolic manifold.
\en
As we argued above, the restriction of $\e_w\in H^1(X_w;\Z)$ to one of the boundary tori, namely $T_w$,  is non-trivial.
It is well known that in each of the first three cases, this would imply that
$\e_w$ is a fibered class. Hence  $X_w$ must be hyperbolic.
This concludes the proof of the claim. 

In the following, given a vertex $v$ with $\e_v$ non-zero,  we denote by $d_v\in \N$ the divisibility of $\e_v\in H^1(X_v;\Z)$. For all other vertices we write $d_v=0$.

\begin{claim}
There exists a minimal genus Seifert surface $\S$ for $K$ with the following properties:
\bn
\item $\S$ intersects each $T_e$ transversally;
\item each intersection $\S\cap T_e$ consists of a possibly empty union of parallel, non-null-homologous curves;
\item for each $v$ with $d_v\ne 0$ the surface $\S_v:=\S\cap X_v$ is the union of $d_v$ parallel copies of a surface $\S_v'$.
\en
\end{claim}

For each $v$ with $d_v\ne 0$ we pick a properly embedded Thurston norm-minimizing surface $\S_v'$ that represents $\frac{1}{d_v}\e_v$.
After possibly gluing in annuli and disks, we may assume that at each boundary torus $T$ of $N_v$, all the components of $\S_v'\cap T$ are parallel as oriented curves
and no component of $\S_v'\cap T$ is null-homologous.
We now pick a tubular neighborhood $\S_v'\times [-1,2]$ of $\S_v'$
and we denote by $\S_v$ the union of $\S_v'\times r_i$ where $r_i=\frac{i}{d_v}$ with $i=0,\dots,d_v-1$.
For each $v$ with $d_v=0$ we denote by $\S_v'=\S_v$ the empty set.

The surfaces $\S_v$ are chosen such that at each JSJ torus the boundary curves are parallel.
Since at a JSJ edge the adjacent surfaces have to represent the same homology class,  at each JSJ torus
the adjacent surfaces have exactly the same number of boundary components
which furthermore represent the same homology class in the JSJ torus. After an isotopy in the neighborhood of the tori we can therefore glue the surfaces $\S_v$ together to obtain a properly embedded surface $\S$.
Since the Thurston norm is  linear on rays, it follows from \cite[Proposition~3.5]{EN85} that $\S$ is a connected Thurston norm-minimizing surface representing $\e$. By construction,
the intersection of $\S$ with $\partial X$ consists of one curve, which is necessarily a longitude for $K$. We thus see that $\S$ is indeed a genus-minimizing Seifert surface for $K$.
It is now clear that $\S$  has the desired properties.
This concludes the proof of the claim.

Recall that $\e_w\in H^1(X_w;\Z)$ is not a fibered class. By the discussion at the beginning of this section, this implies that
$\frac{1}{d_w}\e_w$ is also not a fibered class, and so $\S_w'$ is not a fiber surface.

We pick a base point $p_w$ on $\S_w'=\S_w'\times 0$, which is then also a base point for $X_w$.
It follows from Theorem \ref{thm:extendfreehyp} that there exists an element $g\in \pi_1(X_w\sm \S_w'\times (0,1),p_w)$ such that the subgroup of $\pi_1(X_w\sm \S_w'\times (0,2],p_w)$ generated by $\pi_1(\S_w',p_w)$ and $g$ is in fact the free product of $\pi_1(\S_w',p_w)$ and $\ll g \rr $.
It now remains to prove the following claim.

\begin{claim}
The subgroup of $\pi_1(X,p_w)$ generated by $\pi_1(\S,p_w)$ and $g$ is the free product of $\pi_1(\S,p_w)$ and $\ll g \rr $.
\end{claim}

We may pick an oriented simple closed curve $c$ in $X_w\sm \S_w'\times (0,2]$ that intersects $\S_w'=\S_w'\times 0$ in precisely the base point $p_w$
and that represents $g\in \pi_1(X_w\sm \S_w'\times (0,2],p_w)$. Note that $\pi_1(\S\cup c,p_w)$ is precisely the free product of $\pi_1(\S,p_w)$ and $\ll g \rr $. 
It thus suffices to show that the inclusion-induced map
\[ \pi_1(\S\cup c,p_w)\to \pi_1(X,p_w) \]
is injective.
\begin{figure}[h]
\begin{center}
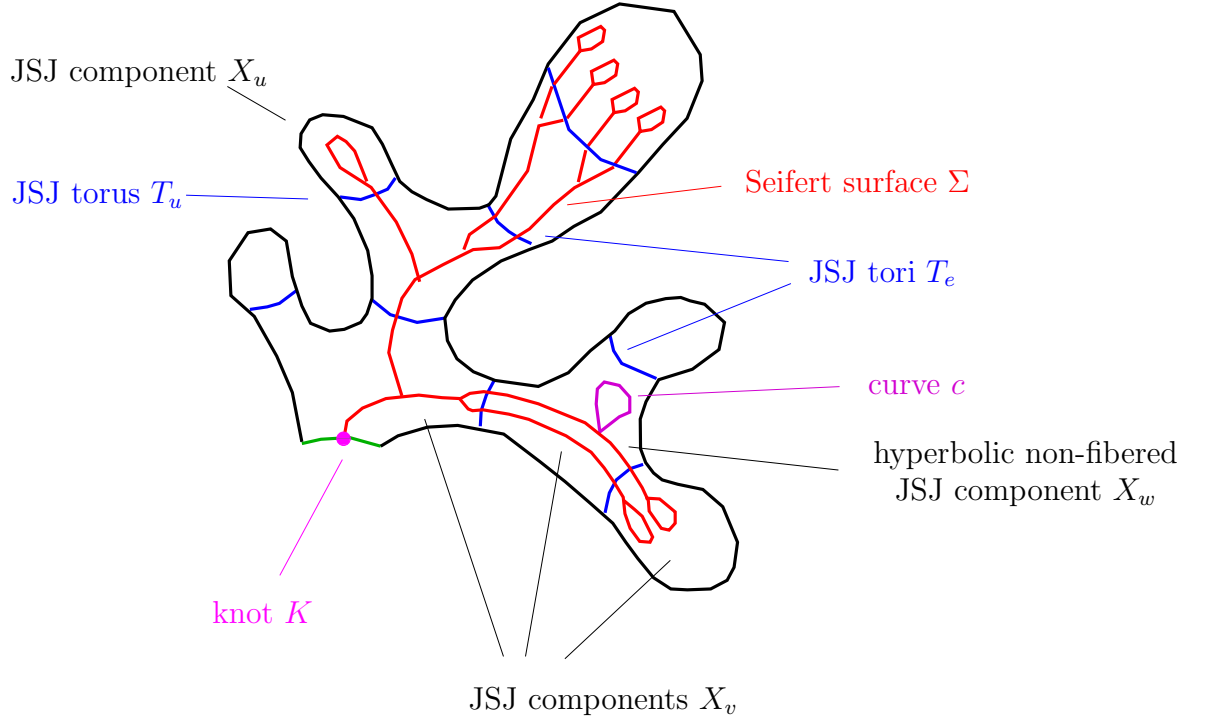
\caption{Schematic picture for the Seifert surface $\S$ and the curve $c$.}
\end{center}
\end{figure}

 Let $h$ be an element in the kernel of this map.
We pick a representative curve $d$ which intersects the JSJ tori transversally. We will show that $h$ represents the trivial element 
in $\pi_1(\S\cup c,p_w)$ by  induction on
\[ n(d):=\sum_{v\in V} \#\mbox{components of $d\cap X_v$}.\]

If $n(d)=1$,  then $d$ lies in the component of 
$(\S\cup w)\cap X_w=\S_w\cup c$ that contains $p_w$. 
Then $c$ lies completely in $\S'_w\cup c$.
But the map $\pi_1(\S'_w\cup c,p_w)=\pi_1(\S'_w,p_w)*\ll g\rr\to \pi_1(X_w)$ is injective,
and the map $\pi_1(X_w)\to \pi_1(X)$ is also injective. It thus follows that $h$ is the trivial element.

We now consider the case that $n:=n(d)>1$. We then think of $\pi_1(X)$ as the fundamental group of the graph of groups $\pi_1(X_v)$.
We can view the curve $d$ as a concatenation of curves $d_1,\dots,d_n$ such that each curve $d_i$ lies completely in some $X_{u}$.
Recall that we assume  that $d$ represents the trivial element. A standard argument in the theory of fundamental groups of graph of groups (see e.g. \cite{He87}) implies that there exists a $d_i$ with the following two properties:
\bn
\item the two endpoints of $d_i$ lie on the same boundary torus $T$ of some $X_{u}$,
\item $d_i$ is homotopic in $X_{u}$ rel endpoints to a curve $s_i$ that lies completely in $T$.
\en
Note that the two endpoints of $d_i$ lie on $T\cap \S_u$.
In fact we can prove a stronger statement.

\begin{claim}
The two endpoints of $d_i$ lie on the same component of $T\cap \S_u$.
\end{claim}

We first make the following observation.
Let $S$ be a properly oriented embedded surface $S$ in an oriented 3-manifold $M$ and let $a$ be an oriented embedded arc  that does 
not intersect $S$ at the endpoints.
 We can then associate to $S$ and $a$ the algebraic  intersection number $S\cdot a\in \Z$, which has in particular  the following two properties:
\bn
\item for any properly oriented embedded arc $b$ homotopic to $a$ rel base points we have $S\cdot a=S\cdot b$,
\item if $a$ lies completely in a boundary component $B$ of $M$, then $S\cdot a$ equals the algebraic intersection number of the oriented curve $\partial S$ with the oriented arc $a$ in $B$.
\en

We now turn to the proof of the claim. We first note that there exists a homeomorphism $r\colon X_u\to X_u$ 
which is the identity on $X_u\sm \S_u'\times (-1,2)$, which 
has the property that for any $x\times t$ with $x\in \S_u'$ and $t\in [0,1]$ we have 
\[ f(x\times t)=x\times (t-\frac{1}{2d_v})\]
and which is isotopic to the identity on $X_u$. 
More informally,  $r$ is a map that pushes everything on
 $\S\times [0,1]$ slightly to the left.
Note that $r$ pushes everything on $\S_u$ off $\S_u$.
Furthermore, if $u=w$, then the intersection of $r(\S_w\cup c)$ 
with $\S_w$ is also empty. 

Since $s_i$ and $d_i$ are homotopic rel base points and since $r$ is homotopic to the identity, the curves $r(s_i)$ and $r(d_i)$ are homotopic rel base points. It follows from the above that $\S_u\cdot r(s_i)=\S_u\cdot r(d_i)$.
But the latter is clearly zero, since $r(d_i)$ does not intersect $\S_u$. 
We now conclude that $\partial \S_u \cdot r(s_i)=\S_u\cdot r(s_i)=0$.
Since the curves $\partial \S_u\cap T$ are all parallel it now follows that $r(s_i)$ does not intersect $\S_u\cap T$ at all.
But this means  that the two endpoints of $s_i$, and thus also the two endpoints of $d_i$, have to lie on the same component of $T\cap \S_u$.
This concludes the proof of the claim.

We then make the following claim.

\begin{claim}
The curve $d_i$ is homotopic in $X_{u}$ rel end points to a curve $d_i'$ that lies completely in $T\cap \S_u$.
\end{claim}

By the previous claim we know that  the two endpoints of $d_i$ lie on the same component of $T\cap \S_u$.
We denote the initial point of $d_i$ by $P$, and the terminal point by $Q$.
We denote by $r$ the component of $\partial \S_u$ that contains $P$. We endow $r$ with an orientation.
Note that $r$ is homologically essential on $T$. The curve $r$ thus defines a subsummand $\ll r\rr$ of $\pi_1(T,P)\cong \Z^2$.

We also pick a curve $t_i$ in $T\cap X_u$ from $P$ to $Q$.
The concatenation $s_it_i^{-1}$ lies in $T$, and also lies in $(\S\cup c)\cap X_u$.
The curve $s_it_i^{-1}$ thus represents an element in $\pi_1((\S\cup c)\cap X_u,P)\cap \pi_1(T,P)$.
But the group $\pi_1((\S\cup c)\cap X_u,P)$ is free (regardless of whether $c$ lies on the $P$-component of $(\S\cup c)\cap X_u$ or not) whereas $\pi_1(T,P)\cong \Z^2$.
The two groups thus intersect in an infinite cyclic subgroup. Furthermore, the intersection contains the subsummand $\ll r\rr$. It follows that the intersection equals $\ll r\rr$.
In particular, $s_i^{-1}t_i$ is homotopic rel $P$ to $r^k$ for some $k$. It now follows that relative to the end points we have the following homotopies:
\[ d_i\sim d_is_i^{-1}s_i\sim s_i\sim s_it_i^{-1}t_i\sim r^kt_i.\]
But the curve $d_i':=r^kt_i$ lies completely in $T\cap \S_u$.
This concludes the proof of the claim.

We can thus replace $d=d_1\dots d_{i-1}d_id_{i+1}\dots d_l$ by $d_1\dots d_{i-1}d_i'd_{i+1}\dots d_l$
and push $d_i'$ slightly into the adjacent JSJ component of $X$. We have found a representative of $h$ of smaller length than $d$.
The claim that $h$ represents the trivial element now follows by induction.

This concludes the proof  that the subgroup of $\pi_1(X\sm \S\times (0,2],p_w)$ generated by $ \pi_1(\S,p_w)$ and $g$ is the free product of $\pi_1(\S,p_w)$ and $\ll g \rr $.
We are therefore done with the proof of Theorem \ref{thm:infinitecyclic}.
\end{proof}

\section{Comparison with Stallings's fibering criterion}\label{section:stallings}

Let $K$ be a knot. Recall that we denote by $\e_K\co \pi(K)\to \Z$ the unique epimorphism that sends the oriented meridian to 1.
Stallings \cite{St62} proved the following theorem.

\begin{theorem}\label{thm:st62}
If $K$ is not fibered, then  $\ker(\e_K)$ is not finitely generated.
\end{theorem}

It follows from Lemma \ref{lem:splitkerfg} that if 
  $\ker(\e_K)$ is finitely generated, then there exists precisely one group $B$ such that $\pi(K)$ splits over $B$.
Thus Stalling's theorem follows as a consequence of either Theroem \ref{thm:fibsplit} or Theorem \ref{thm:splitfreelarge}.

On the other hand, a group $\pi$ with an epimorphism $\e\co \pi\to \Z$ such that $\ker(\e)$ is not finitely generated may still split over a unique group.   
The Baumslag-Solitar group, the semidirect product $\Z\ltimes \Z[\frac{1}{2}]$ where $n\in \Z$ acts on $\Z[\frac{1}{2}]$ by multiplication by $2^n$, has abelianization $\Z$. The kernel of the abelianization $\e:\pi\to \Z$ is the infinitely generated subgroup $\Z[\frac{1}{2}]$. Since every finitely generated subgroup of
$\Z[\frac{1}{2}]$ is isomorphic to $\Z$, $\Z\ltimes \Z[\frac{1}{2}]$ splits only over subgroups isomorphic to $\Z$. (In fact, any two splittings are easily seen to be strongly equivalent.) 
This shows that the conclusions of Theorems \ref{thm:fibsplit} and \ref{thm:splitfreelarge} are indeed stronger than the conclusion of Theorem \ref{thm:st62}.

Stallings's fibering criterion has been generalized in several other ways.
For example, if $K$ is not fibered, then $\ker(\e)$ can be written neither as a descending nor as an ascending HNN-extension \cite{BNS87}, $\ker(\e)$ admits uncountably many subgroups of finite index (see \cite[Theorem~5.2]{FV12c}, \cite{SW09a} and \cite[Theorem~3.4]{SW09b}),
the pair $(\pi(K),\e_K)$ has `positive rank gradient' (see \cite[Theorem~1.1]{DFV12})
and $\ker(\e_K)$ admits a finite index subgroup which is not normally generated by finitely many elements (see \cite[Theorem~5.1]{DFV12}).

\section{Proof of Theorem \ref{mainthm}}

In this section we will prove Theorem \ref{mainthm}, i.e. we will show that if $K$ is a knot, then $\pi(K)$ does not split over a group of rank less than $2g(K)$. 
We will first give a `classical' proof for genus-one knots
before we provide the proof for all genera.

\subsection{Genus-one knots}\label{section:genusone}

In this subsection we prove:

\begin{theorem}\label{thm:genus1}
If $K$ is a genus-one knot, then $\pi(K)$ does not split 
over a free group of rank less than two.
\end{theorem}

The main ingredients in the proof are  
 two classical results from 3-manifold topology.
First, we recall the statement of the Kneser Conjecture,  which
 was first proved by Stallings \cite{St59} in the closed case, and by
Heil \cite[p.~244]{Hei72} in the bounded case.

\begin{theorem} \textbf{\emph{(Kneser Conjecture)}}\label{thm:kneserconj}
Let $N$ be a   $3$-mani\-fold with incompressible boundary.
If there exists an isomorphism $\pi_1(N)\cong \G_1*\G_2$, then there exist compact, orientable $3$-manifolds $N_1$ and $N_2$
with $\pi_1(N_i)\cong \G_i$, $i=1,2$ and $N\cong N_1\# N_2$.
\end{theorem}

In the following, we say that a properly embedded 2-sided annulus $A$ in a 3-manifold $N$ is \emph{essential} if the inclusion map $A \hookrightarrow N$ induces a $\pi_1$-injection and if $A$ is not properly homotopic into $\partial N$.  The second classical result we will use is the following,
which is a direct consequence of a theorem of Waldhausen \cite{Wal68b} (see Corollary 1.2(i) of \cite{Sco80}).

\begin{theorem}\label{thm:annulus}
Let $N$ be an irreducible 3-manifold with incompressible boundary.
If $\pi_1(N)$ splits over $\Z$, then $N$ contains an essential, properly embedded 2-sided annulus.
\end{theorem}

We turn to the proof of Theorem \ref{thm:genus1}.

\begin{proof}[Proof of Theorem \ref{thm:genus1}]
Let $K$ be a genus-one knot. Since $K$ is non-trivial, the Loop Theorem implies that  $\partial X(K)$ is incompressible.
Since knot complements are prime 3-manifolds, it now follows from the Kneser Conjecture that $\pi(K)$ can not split over the trivial group, i.e. $\pi(K)$ 
cannot split over a free group of rank zero.

Now suppose that $J$ is a non-trivial knot such that  $\pi(J)$ splits over a free group of rank one, that is, over a group isomorphic to $\Z$.
From   Theorem \ref{thm:annulus} we  deduce that  $X(J)$ contains an essential,
properly embedded, 2-sided annulus $A$. Lemma 2 of \cite{Ly74a} (an immediate consequence of \cite{Wal68a}) implies that the knot $J$ is either a composite or a nontrivial cable knot. If $J$ is a composite knot, then it follows from the additivity of the knot genus
(see, for example, \cite[p.~124]{Ro90}) that the genus of $J$ is at least two. 
Moreover, a well-known result of Schubert \cite{Sct53} (see Proposition 2.10 of \cite{BZ85}) implies that the genus of any cable knot is greater than one. Thus in both cases we see that $g(J)\geq 2$. 

We now see that for the genus-one knot $K$ the group $\pi(K)$ cannot split over  a free group of rank one.
 \end{proof}

\subsection{Wada's invariant}\label{section:wada}

For the proof of Theorem \ref{mainthm3} we will need Wada's invariant,
which is also known as the twisted Alexander polynomial or the twisted Reidemeister torsion of a knot.

We introduce the following convention. If $\pi$ is a group and $\g:\pi\to \gl(k,R)$ a representation over a ring,
then we denote by $\g$ also the  $\Z$-linear extension of $\g$ to  a map $\Z[\pi] \to M(k,R)$.
Furthermore, if $A$ is a matrix over $\Z[\pi]$ then we denote by $\g(A)$ the matrix given by applying
$\g$ to each entry of $A$.

Let $\pi$ be a group, $\e\co \pi\to \Z$ an epimorphism, and $\a \co \pi\to \gl(k,\C)$ a representation.
First note that $\a$ and $\e$ give rise to a tensor representation
\[ \begin{array}{rcl} \a \otimes \e \co \pi &\to & \gl(k,\ct) \\
g&\mapsto & t^{\e(g)}\cdot \a(g).\end{array} \]
Now let
\[ \pi=\langle g_1,\dots,g_{k} \,|\, r_1,\dots,r_{l}\rangle\]
be a presentation of $\pi$. By adding trivial relations if necessary, we may assume that $l\geq k-1$.
We denote by $F_{k}$ the free group with generators $g_1,\dots,g_{k}$.
Given $j\in \{1,\dots,k\}$ we denote by  $\frac{\partial }{\partial g_j}\co \Z[F_{k}]\to \Z[F_{k}]$ the Fox derivative with respect to $g_j$, i.e. the unique $\Z$-linear map such that
\begin{eqnarray*}
\frac{\partial g_i}{\partial g_j}&=&\delta_{ij},\\
\frac{\partial uv}{\partial g_j}&=&\frac{\partial u}{\partial g_j}+u\frac{\partial v}{\partial g_j}
\end{eqnarray*}
for all $i,j \in \{1,\dots,k\}$ and $u,v\in F_{k}$.
We denote by
\[M:=\left(\frac{\partial r_i}{\partial g_j}\right)\]
the $l\times k$-matrix over $\Z[\pi]$ of all the Fox derivatives of the relators.
Given subsets $I=\{i_1,\dots,i_r\}\subset \{1,\dots,k\}$ and
$J=\{j_1,\dots,j_s\}\subset \{1,\dots,l\}$ we denote by $M_{J,I}$ the  matrix formed by deleting the columns
$i_1,\dots,i_r$ and by deleting the rows $j_1,\dots,j_s$ of $M$.

Note that there exists at least one $i\in \{1,\dots,k\}$ such that $\e(g_i)\ne 0$. It follows that
\[ \det((\a\otimes \e)(1-g_i))=\det\left(\id_k-t^{\e(g_i)}\a(g_i)\right)\ne 0.\]
We define
\[ Q_i:=\mbox{gcd}\{\det((\a\otimes \e)(M_{J,\{i\}}))\,|\, J\subset \{1,\dots,l\} \mbox{ with }|J|=l+1-k\}.\]
(Note that each $M_{J,\{i\}}$ is a $(k-1)\times (k-1)$-matrix.) 
It is worth considering the special case that  $l=k-1$; that is, the case of a presentation of deficiency one. Then the only choice for $J$ is the empty set, and hence
\[ Q_i=\det((\a\otimes \e)(M_{\emptyset,\{i\}})).\]
Wada \cite{Wad94} introduced the following invariant of the triple $(\pi,\e,\a)$.
\[ \Delta_{\pi,\e}^\a :=Q_i\cdot \det((\a\otimes \e)(1-g_i))^{-1}\in \C(t).\]
A priori, Wada's invariant depends on the various choices we made.
The following theorem proved by Wada \cite[Theorem~1]{Wad94} shows that the indeterminacy is well controlled.

\begin{theorem}\label{thm:wadadefined}
Let $\pi$ be a group, let $\e\co \pi\to \Z$ be an epimorphism, and let $\a \co \pi\to \gl(k,\C)$
 be a representation.
Then  $\Delta_{\pi,\e}^\a$ is well-defined up to multiplication by a factor of the form $\pm t^kr$, where $k\in \Z$ and
$r\in \C^*$.
\end{theorem}

\medskip

Finally,  let $K\subset S^3$ be a knot and let $\a \co \pi(K)\to \gl(k,\C)$ be a representation.
As before, we denote by $\e\co \pi(K)\to \Z$ the epimorphism that sends the oriented meridian of $K$ to 1. We write
\[ \Delta_K^\a=\Delta_{\pi,\e}^\a.\]
If $\a\co \pi(K)\to \gl(1,\C)$ is the trivial one-dimensional representation, then Wada's invariant is determined by the classical Alexander polynomial $\Delta_K$. More precisely, we have
\[ \Delta_K^\a=\frac{\Delta_K}{1-t}.\]
Wada's invariant equals the twisted Reidemeister torsion of a knot, and is closely related
to the twisted Alexander polynomial of a knot, which was first introduced by Lin \cite{Lin01}.
We refer to \cite{Ki96,FV10} for more details about Wada's invariant, its interpretation as twisted Reidemeister torsion and its relationship to twisted Alexander polynomials.

\subsection{Proof of Theorem \ref{mainthm}}\label{section:proof}

Before we provide the proof of Theorem \ref{mainthm} we need to introduce two more definitions.
First, given a non-zero polynomial $p(t)=\sum_{i=r}^s a_it^i\in \ct$ with $a_r\ne 0$ and $a_s\ne 0$, we write
\[ \deg(p(t))=s-r.\]
If $f(t)=p(t)/q(t)\in \C(t)$ is a non-zero rational function, we write
\[ \deg(f(t))=\deg(p(t))-\deg(q(t)).\]
Note that if Wada's invariant of a triple $(\pi,\e,\a)$ is non-zero, then the degree of Wada's invariant $\Delta_{\pi,\e}^\a$ is well defined.
\medskip

We can now formulate the following  theorem.

\begin{theorem}\label{thm:technical}
Let $\pi$ be a group and let
\[ f:\pi \to \ll A,t\,|\, f(B)=tBt^{-1}\rr\]
be a splitting.
We denote by  $\e\co \ll A,t\,|\, f(B)=tBt^{-1}\rr\to \Z$ the canonical epimorphism  which is given by
$\e(t)=1$ and $\e(a)=0$ for $a\in A$.
If $\a\co \pi\to \gl(k,\C)$ is a representation such that $\Delta_{\pi,\e}^\a\ne 0$, then
\[ \deg \Delta_{\pi,\e}^\a \leq k(\rank(B)-1).\]
\end{theorem}

In \cite{FKm06} (see also \cite{Fr12}) it was shown
that if $K$ is a knot and $\a\co \pi(K)\to\gl(k,\C)$ is a representation such that $\Delta_K^\a\ne 0$, then
\be \label{equ:genuslowerbound} \deg \Delta_K^\a \leq k(2\op{genus}(K)-1).\ee
In light of the discussion in Section \ref{section:splitk}, we can view Theorem \ref{thm:technical} as a generalization of (\ref{equ:genuslowerbound}).

\begin{proof}
Let $\pi$ be a group and let
\[ \pi =\ll g_1,\dots,g_k,t\,|\, r_1,\dots,r_l,\varphi(b)=tbt^{-1}\mbox{ for all }b\in B\rr\]
be a splitting, where  $\varphi\co B\to A$ is a monomorphism and $B$ is a rank-$d$ subgroup of  $A=\ll g_1,\dots,g_k,t\,|\, r_1,\dots,r_l\rr$.
We pick generators $x_1,\dots,x_{d}$ for $B$. Note that
\[ \ba{ll} &\ll g_1,\dots,g_k,t\,|\, r_1,\dots,r_l,\varphi(b)=tbt^{-1}\mbox{ for all }b\in B\rr\\
=&\ll g_1,\dots,g_k,t\,|\, r_1,\dots,r_l,\varphi(x_1)^{-1}tx_1t^{-1},\dots,\varphi(x_{d})^{-1}tx_{d}t^{-1}\rr.\ea\]  We write $K:=\ker(\e)$.

We denote by $M$
the $(l+{d})\times (k+1)$-matrix over $\Z[\pi]$ that is given by all the Fox derivatives of the relators. We make the following observations.
\bn
\item The relators $r_1,\dots,r_l$ are words in $g_1,\dots,g_k$. The Fox derivatives of the $r_i$ with respect to the $g_j$ thus lie in $\Z[K]$.
\item For any $i\in \{1,\dots,k\}$ and $j\in \{1,\dots,{r}\}$ we have
\[ \frac{\partial}{\partial g_i}\left(\varphi(x_j)^{-1}tx_jt^{-1}\right)=\frac{\partial}{\partial g_i}\left(\varphi(x_j)^{-1}\right)+\varphi(x_j)^{-1}t\frac{\partial}{\partial g_i}x_j.\]
The same argument as in (1) shows that the first term lies in $\Z[K]$, and one can similarly see that the second term is of the form $t\cdot g$, where $g\in \Z[K]$.
\en
Thus $M_{\emptyset,\{k+1\}}$, the matrix obtained from $M$ by deleting the $(k+1)$-st column, is of the form
\[ M_{\emptyset,\{k+1\}}=P+tQ,\]
where $P$ and $Q$ are matrices over $\Z[K]$, and where all but the last ${d}$ rows of $Q$ are zero.

Let $\a\co \pi\to \gl(k,\C)$ be a representation and $J\subset \{1,\dots,d+l\}$ a subset  with $|J|=d+l-k$.
It follows from the above that
\[ M_{J,\{k+1\}}=P_J+tQ_J,\]
where $P_J$ and $Q_J$ are matrices over $\Z[K]$ and where at most  $d$ rows of $Q_J$ are non-zero.
We then see  that
\[ \det((\a\otimes \e)(M_{J,\{k+1\}}))=\det(\a(P_J)+t\a(Q_J)),\]
where at most  $kr$ rows of $\a(Q_J)$ are non-zero.
If $\det(\a(P_J)+t\a(Q_J))$ is non-zero, then it follows from an elementary argument
that
\[ \deg(\det(\a(P_J)+t\a(Q_J)))\leq kr.\]
We now consider
\[ Q:=\mbox{gcd}\{\det((\a\otimes \e)(M_{J,\{k+1\}}))\,|\, J\subset \{1,\dots,l\} \mbox{ with }|J|=d+l-k\}.\]
By the above, if $Q\ne 0$, then $\deg(Q)\leq kr$.

Since $\e(t)=1$,
\[ \Delta_{\pi,\e}^\a =Q\cdot \det((\a\otimes \e)(1-t))^{-1}=Q\cdot \det(\id_k-\a(t)t)^{-1}\in \C(t).\]
Finally, we suppose that $\Delta_{\pi,\e}^\a\ne 0$. By the above, this implies that $Q\ne 0$.
In particular, we see that
\[ \ba{rcl} \deg(\Delta_{\pi,\e}^\a)&=&\deg\left(Q\cdot \det(\id_k-\a(t)t))\right)\\
&=&\deg(Q)-\deg(\det(\id_k-\a(t)t))\\
&=&\deg(Q)-k\\
&\leq &kr-k=k(\rank{B}-1).\ea\]
This concludes the proof of the theorem.
\end{proof}

The last ingredient in the proof of Theorem \ref{mainthm} is the following result from \cite{FV12a}. The proof of the theorem builds on the virtual fibering theorem of Agol \cite{Ag08} (see also \cite{FKt12}), which applies for knot complements by the work of Liu \cite{Liu11}, Przytycki-Wise \cite{PW11,PW12a} and Wise \cite{Wi09,Wi12a,Wi12b}.

\begin{theorem}\label{thm:fv12a}
Let $K$ be a knot. Then there exists
a representation $\a\co \pi(K)\to \gl(k,\C)$ such that  $\Delta_K^\a\ne 0$ and such that
\[ \deg \Delta_K^\a=k(2g(K)-1).\]
\end{theorem}

In \cite[Theorem~1.2]{FV12a} an analogous statement is formulated  for twisted Reidemeister torsion instead of Wada's invariant.
The theorem, as stated, now follows from the interpretation (see, for example, \cite{Ki96,FV10}) of Wada's invariant as twisted Reidemeister torsion.

We can now formulate and prove the following result, which is equivalent to Theorem \ref{mainthm}.

\begin{theorem}\label{mainthm2}
Let $K$ be a knot. If $\pi(K)$ splits over a group $B$, then $\rk(B)\geq 2g(K)$.
\end{theorem}

\begin{proof}
Let $K$ be a knot and let
\[ f\co \pi(K) \to \pi=\ll A,t\,|\, \varphi(B)=tBt^{-1}\rr\]
be an isomorphism. We denote by  $\e\co \ll A,t\,|\, \varphi(B)=tBt^{-1}\rr\to \Z$ the canonical epimorphism  which is given by
$\e(t)=1$ and $\e(a)=0$ for $a\in A$.

Note that $\e\circ f\co \pi(K)\to \Z$ is an epimorphism. In particular, it sends the meridian to either $1$ or $-1$. By possibly changing the orientation of the knot, we can assume that $\e\circ f\co \pi(K)\to \Z$ sends the meridian to  $1$.
By Theorem \ref{thm:fv12a}, there exists
a representation $\a\co \pi(K)\to \gl(k,\C)$ such that $ \Delta_K^\a\ne 0$ and such that
\[ \deg \Delta_K^\a=k(2g(K)-1).\]
By definition, we have
\[ \Delta_K^\a=\Delta_{\pi(K),\e\circ f}^\a=\Delta_{\pi,\e}^\a.\]
Theorem \ref{thm:technical} implies that
\[
\rank(B)
\geq  \frac{1}{k}\deg\left( \Delta_{\pi,\e}^\a\right)+1= \frac{1}{k}\deg\left( \Delta_{K}^\a\right)+1=2g(K). \]
\end{proof}

\end{document}